\documentclass[a4paper,11pt]{amsart}
\usepackage{amsmath,amssymb,dsfont,amscd}
\usepackage[utf8]{inputenc} % can use keyboard characters for accented letters
\usepackage{cmap}           % make the PDF file "searchable and copyable"
\usepackage{hyperxmp}       % store XMP metadata in the PDF file
\usepackage{graphicx}
\usepackage{array}
\usepackage{colortbl}
\usepackage[pdfdisplaydoctitle = true,
  colorlinks = true,
  urlcolor = blue,
  citecolor = blue,
  linkcolor = blue,
  pdfstartview = FitH,
  pdfpagemode = UseNone,
  pdfencoding = auto,
  bookmarksnumbered = true]{hyperref} % to make a pdf file with hyperlinks and bookmarks
\usepackage[english]{babel}
\usepackage{float}
\restylefloat{table}
\usepackage{longtable}
\newtheorem{thm}{Theorem}[section]
\newtheorem{cor}[thm]{Corollary}
\newtheorem{lem}[thm]{Lemma}
\newtheorem{prop}[thm]{Proposition}
\theoremstyle{definition}
\newtheorem{defn}[thm]{Definition}
\theoremstyle{remark}
\newtheorem{rem}[thm]{Remark}
\newtheorem{ex}[thm]{Example}
\newcommand{\CC}{\mathbb{C}}    % Complex numbers
\newcommand{\RR}{\mathbb{R}}    % Real numbers
\newcommand{\Ela}{\mathbb{E}\mathrm{la}} % Space of elasticity tensors
\newcommand{\HH}{\mathbb{H}}    % Space of harmonic tensors
\newcommand{\TT}{\mathbb{T}}    % Space of tensors
\newcommand{\Sym}{\mathbb{S}}   % Space of totally symmetric tensors
\newcommand{\VV}{\mathbb{V}}    % vector space
\newcommand{\Sn}[1]{\mathrm{S}_{#1}}        % Space of binary forms
\newcommand{\Hn}[1]{\mathcal{H}_{#1}}       % Harmonic polynomials
\newcommand{\SL}{\mathrm{SL}}   % Special linear group
\newcommand{\OO}{\mathrm{O}}    % Orthogonal group
\newcommand{\SO}{\mathrm{SO}}   % Special orthogonal groups
\newcommand{\id}{\mathrm{I}}    % Identity element of the group

\newcommand{\slc}{\mathfrak{sl}}            % Special linear algebra

\newcommand{\bxi}{\pmb{\xi}}    % vector in C^{2}
\newcommand{\eee}{\pmb{e}}      % vector in R^{3}
\newcommand{\vv}{\pmb{v}}       % vector in R^{3}
\newcommand{\ww}{\pmb{w}}       % vector in R^{3}
\newcommand{\xx}{\pmb{x}}       % vector in R^{3}
\newcommand{\ff}{\mathbf{f}}                % binary form
\newcommand{\bg}{\mathbf{g}}                % binary form
\newcommand{\bj}{\mathbf{j}}                % binary form
\newcommand{\bQ}{\mathbf{Q}}                % binary form
\newcommand{\btheta}{\pmb{\theta}}          % binary form

\newcommand{\rp}{\mathrm{p}}                % polynomial in 3 variables
\newcommand{\rh}{\mathrm{h}}                % harmonic polynomial in 3 variables
\newcommand{\ba}{\mathbf{a}}
\newcommand{\bb}{\mathbf{b}}

\newcommand{\bd}{\mathbf{d}}
\newcommand{\be}{\mathbf{e}}

\newcommand{\bt}{\mathbf{t}}
\newcommand{\bv}{\mathbf{v}}
\newcommand{\bq}{\mathbf{q}}

\newcommand{\bE}{\mathbf{E}}    % elasticity tensor
\newcommand{\bH}{\mathbf{H}}    % harmonic tensor
\newcommand{\bK}{\mathbf{K}}    % third-order symmetric tensor
\newcommand{\bS}{\mathbf{S}}    % totally symmetric tensor or fourth-order symmetric tensor
\newcommand{\bT}{\mathbf{T}}    % generic tensor
\newcommand{\bJ}{\mathbf{J}}    % covariant
\newcommand{\bF}{\mathbf{F}}    % harmonic tensor
\newcommand{\bG}{\mathbf{G}}    % harmonic tensor
\newcommand{\beps}{\boldsymbol{\varepsilon}}

\newcommand{\trans}[3]{\lbrace #1,#2\rbrace_{#3}}

\DeclareMathOperator{\tr}{tr}
\DeclareMathOperator{\Ad}{Ad}

\DeclareMathOperator{\1dot}{\cdot}
\DeclareMathOperator{\2dots}{:}
\DeclareMathOperator{\3dots}{\raisebox{-0.25ex}{\vdots}}
\newcommand{\rdots}[1]{\overset{(#1)}{\cdot}}
\newcommand{\symrdots}[1]{\overset{(#1)}{\underset{s}{\cdot}}}

\newcommand{\norm}[1]{\lVert#1\rVert}    % norm
\newcommand{\set}[1]{\left\{#1\right\}}  % set
\hypersetup{
 pdfauthor = {R. Desmorat, N. Auffray, B. Desmorat, B. Kolev, M. Olive}, % author
 pdftitle = {Generic separating sets for 3D elasticity tensors}, % title
 pdfsubject = {MSC 2010: 74E10 (15A72 74B05)}, % subjet
 pdfkeywords = {Anisotropy; Polynomial invariants; Rational invariants, Separating sets}, % keywords
 pdflang = en, % language : fr, en
 }
\begin{document}

\title[Generic separating sets]{Generic separating sets for 3D elasticity tensors}

\author{R. Desmorat}
\address[Rodrigue Desmorat]{LMT  (ENS Cachan, CNRS, Universit\'{e} Paris Saclay), F-94235 Cachan Cedex, France}
\email{desmorat@lmt.ens-cachan.fr}

\author{N. Auffray}
\address[Nicolas Auffray]{MSME, Universit\'{e} Paris-Est, Laboratoire Mod\'{e}lisation et Simulation Multi Echelle, MSME UMR 8208 CNRS, 5 bd Descartes, 77454 Marne-la-Vall\'{e}e, France}
\email{Nicolas.auffray@univ-mlv.fr}

\author{B. Desmorat}
\address[Boris Desmorat]{Sorbonne Universit\'{e}, UMPC Univ Paris 06, CNRS, UMR 7190, Institut d'Alembert, F-75252 Paris Cedex 05, France \& Univ Paris Sud 11, F-91405 Orsay, France}
\email{boris.desmorat@upmc.fr}

\author{B. Kolev}
\address[Boris Kolev]{LMT  (ENS Cachan, CNRS, Universit\'{e} Paris Saclay), F-94235 Cachan Cedex, France}
\email{boris.kolev@math.cnrs.fr}

\author{M. Olive}
\address[Marc Olive]{LMT  (ENS Cachan, CNRS, Universit\'{e} Paris Saclay), F-94235 Cachan Cedex, France}
\email{marc.olive@math.cnrs.fr}

%\thanks{}%

\date{\today}

% ----------------------------------------------------------------

\begin{abstract}
  We define what is a \emph{generic separating set} of invariant functions (a.k.a. a \emph{weak functional basis}) for tensors. We produce then two generic separating sets of polynomial invariants for 3D elasticity tensors, one made of 19 polynomials and one made of 21 polynomials (but easier to compute) and a generic separating set of 18 rational invariants. As a byproduct, a new
  integrity basis for the fourth-order harmonic tensor is provided.
\end{abstract}

\maketitle

% ----------------------------------------------------------------
\section{Introduction}
\label{sec:intro}
% ----------------------------------------------------------------

Assuming that one could measure the elasticity tensors of two materials, it is a natural question to ask, \emph{if one can decide by finitely many calculations, whether the two materials have identical elastic properties} (are identical as elastic materials), in other words if the two elasticity tensors are related by a rotation. More precisely, two elasticity tensors $\bE_{1}$ and $\bE_{2}$ belonging to the vector space $\Ela$, of fourth order tensors having major and left/right minor indicial symmetries
\begin{equation*}
  E_{ijkl} = E_{ijlk} = E_{klij},
\end{equation*}
define the \emph{same elastic material}, if and only if, there exists a rotation $g\in\SO(3)$ such that
\begin{equation*}
  (\bE_{2})_{ijkl}= g_{ip}g_{jq}g_{kr} g_{ls} (\bE_{1})_{pqrs} \,,
\end{equation*}
a relation that we shall denote by
\begin{equation*}
  \bE_{2} = g \star \bE_{1},
\end{equation*}
and we say then that the two tensors \emph{are in the same orbit}. When such a rotation does not exist, the two tensors describe \emph{different} elastic materials.

Based on the fact that the algebra of invariant polynomials of a linear representation of the rotation group is \emph{finitely generated}~\cite{Hil1993,Stu2008} and \emph{separates the orbits}~\cite[Appendix C]{AS1983}, abstract invariant theory gives an affirmative answer to this question. Nevertheless, calculating explicitly a generating set for this invariant algebra can be an extremely difficult task.

The determination of such a set for the elasticity tensor has a long history, which can be traced back to the work of Betten~\cite{Bet1982,Bet1987}, who obtained some partial results. The question was formulated in rigourous mathematical terms by Boehler et al. in~\cite{BKO1994}, where the link with \emph{invariants of binary forms} was established for the first time. However, the authors did not provide a final answer to the problem. A minimal set of 297 generators for the invariant algebra of the 3D elasticity tensor was finally obtained in 2017, by some of the present authors, in~\cite{OKA2017}, which definitively solved this old problem (see also~\cite{OKDD2018a} for a tensorial expression of these generators, who were first expressed in~\cite{OKA2017} using \emph{transvectants} of binary forms).

Whether this minimal integrity basis can be reduced to obtain a separating set (a.k.a., a functional basis in the mechanical community \cite{WP1964}) of lower cardinality is nevertheless still an open problem. The difficulty is that there is no known general procedure to produce explicit general separating sets whereas there are constructive algorithms to obtain integrity bases~\cite{DK2015,Oli2017}.

There is a huge literature on integrity and functional basis for an $n$-uplet of second-order symmetric tensors or more generally for a family of second-order symmetric tensors and vectors (including, thus, skew-symmetric second-order tensors)~\cite{You1898/99,SR1958/1959,Smi1965,Smi1971,Wan1970a,Wan1970b}. Usually, these functional bases are polynomial~\cite{Wan1969a,Wan1969b,Smi1971}. For higher-order tensors, results are usually sparse or incomplete~\cite{BKO1994,SB1997,OA2014}. Up to the authors best knowledge, nothing is known concerning the elasticity tensor but the 297 invariants of a minimal integrity basis~\cite{OKA2017,OKDD2018}.

Since most materials have no symmetry in practice (they are \emph{triclinic}), their membership to higher-symmetry classes is just a convenient approximation of the reality. Therefore, the notion of separating set/functional basis can be weakened again, in order to reduce its cardinal. To be more specific, the notion of \emph{weak separating set} --- also known as a \emph{weak functional basis} --- has been formulated in~\cite{BKO1994}, in the sense that they separate only \emph{generic} tensors (defined rigorously in Section~\ref{sec:separating-sets}, using Zariski topology). In~\cite{BKO1994}, Boehler, Kirillov and Onat produced a weak separating set of 39 polynomial invariants for $\bE\in \Ela$.

In the present paper, by formulating slightly different genericity conditions, we produce a weak separating set of \emph{21 polynomial invariants} for the elasticity tensor. This result, formulated as Theorem~\ref{thm:main}, is our main theorem. Moreover, translating results on rational invariants of the binary form of degree 8 by Maeda in~\cite{Mae1990}, we can shorten this number to 19 (Corollary \ref{cor:main-19}), but the corresponding polynomial invariants are more complicated. We can also deduce a set of 18 \emph{rational invariants} which separate generic elasticity tensors (Corollary \ref{cor:main-18}).

The paper is organized as follows. In section~\ref{sec:integrity-basis}, we recall basic notions on \emph{integrity basis} and produce a new minimal integrity basis for $\HH^{4}(\RR^3)$, the space of fourth-order harmonic tensors. In section~\ref{sec:separating-sets}, we introduce various definitions of \emph{separating sets} and formulate rigorously the concept of \emph{genericity}. Formulations of the main result, some corollaries and their proofs are provided in Section~\ref{sec:seps-elasticity}. The mathematical material needed to understand the link between \emph{invariant theory of binary forms} and \emph{invariant theory of harmonic tensors} is recalled in~Appendix~\ref{sec:rational-invariants}. A set of 18 \emph{rational invariants} which separate generic fourth-order harmonic tensors is then provided in~Appendix~\ref{sec:Maeda-invariants} by translating \emph{Maeda invariants}~\cite{Mae1990} into invariants of the fourth-order harmonic tensor.

% ----------------------------------------------------------------
\subsection*{Notations}
\label{subsec:notations}
% ----------------------------------------------------------------

We denote by $\TT^{n}(\RR^3)$, the space of $n$th-order tensors on $\RR^{3}$ and by $\Sym^{n}(\RR^3)$, the subspace of totally symmetric tensors of order $n$. A traceless tensor $\bH \in \Sym^{n}(\RR^3)$ is called an \emph{harmonic tensor} and the space of $n$th-order harmonic tensors is noted $\HH^{n}(\RR^{3})$.

The total symmetrisation of a tensor $\bT\in \TT^{n}(\RR^3)$ is the tensor $\bT^{s}\in \Sym^{n}(\RR^3)$, defined by
\begin{equation*}
  (T^{s})_{i_{1}\dotsc i_{n}} =  \frac{1}{n!}\sum_{\sigma \in \mathfrak{S}_{n}} T_{i_{\sigma(1)} \dotsc i_{\sigma(n)}},
\end{equation*}
where $\mathfrak{S}_{n}$ is the permutation group over $n$ elements.

The \emph{symmetric tensor product} between two totally symmetric tensors $\bS^{k}\in \Sym^{n_{k}}(\RR^3)$ is defined as
\begin{equation}\label{eq:symmetric-tensor-product}
  \bS^{1} \odot \bS^{2} := (\bS^{1} \otimes \bS^{2})^{s} \in \Sym^{n_{1}+n_{2}}(\RR^3).
\end{equation}

The \emph{$r$-contraction} between two tensors $\bT^{k}\in \TT^{n_{k}}(\RR^3)$ is defined (in an orthonormal basis) as
\begin{equation*}
  (\bT^{1} \overset{(r)}{\cdot} \bT^{2})_{i_{1} \dotsc i_{n_{1}-r}j_{r+1} \dotsc j_{n_{2}}} := T^{1}_{i_{1}\dotsc i_{n_{1}-r}k_{1}\dotsc k_{r}} \, T^{2}_{k_{1}\dotsc k_{r}j_{r+1} \dotsc j_{n_{2}}}.
\end{equation*}
In particular, we get
\begin{equation*}
  \begin{array} {ll}
    (\ba\1dot\bb)_{ij} = a_{ik}b_{kj},        & \ba\2dots\bb=a_{ij}b_{ij},                 \\
    (\bH\2dots\ba)_{ij} = H_{ijkl}a_{kl},     & (\bH \2dots \bK)_{ijkl} = H_{ijpq}K_{pqkl} \\
    (\bH \3dots \bK)_{ij} = H_{ipqr}K_{pqrj}. &
  \end{array}
\end{equation*}
where $\ba, \bb$ are two second-order tensors and $\bH, \bK$, two fourth-order tensors. The usual abbreviations
$\ba^2 = \ba\1dot\ba$, $\ba\bb = \ba\1dot\bb$ and $\bH^{2}=\bH\2dots\bH$ shall also be used.

The \emph{symmetric $r$-contraction} between two totally symmetric tensors $\bS^{k}\in \Sym^{n_{k}}(\RR^3)$ is defined as
\begin{equation}\label{eq:symmetric-r-contraction}
  \bS^{1} \symrdots{r} \bS^{2} := (\bS^{1} \overset{(r)}{\cdot} \bS^{2})^{s} \in \Sym^{n_{1}+n_{2}-2r}(\RR^3),
\end{equation}

The \emph{generalized cross product} between two totally symmetric tensors $\bS^{k}\in \Sym^{n_{k}}(\RR^3)$, which has been introduced in~\cite{OKDD2018a}, is defined as
\begin{equation}\label{eq:cross-product}
  \bS^{1} \times \bS^{2} := (\bS^{2} \cdot \beps \cdot \bS^{1})^{s}\in \Sym^{n_{1}+n_{2}-1}(\RR^3),
\end{equation}
where $\beps$ is the third-order Levi-Civita tensor.

The \emph{leading harmonic part} $\bS^{\prime} \in \HH^{n}(\RR^{3})$ of a totally symmetric tensor $\bS \in \Sym^{n}(\RR^{3})$ means the harmonic part of highest order of $\bS$ in its harmonic decomposition (see~\cite[Proposition 2.8]{OKDD2018}, where it was noted $\bS_{0}$ rather than $\bS^{\prime}$).

The \emph{harmonic product} between two harmonic tensors $\bH^{k}\in \HH^{n_{k}}(\RR^3)$, which has been introduced in~\cite{OKDD2018}, is defined as
\begin{equation}\label{eq:harmonic-product}
  \bH^{1} \ast \bH^{2} := (\bH^{1} \odot \bH^{2})^{\prime} \in \HH^{n_{1}+n_{2}}(\RR^{3}).
\end{equation}

% ----------------------------------------------------------------
\section{Integrity basis}
\label{sec:integrity-basis}
% ----------------------------------------------------------------

In this paper, we consider a \emph{linear representation} $\VV$ of the 3-dimensional orthogonal group $\OO(3)$. This means that we have a mapping
\begin{equation*}
  (g,\vv) \mapsto g \star \vv, \qquad \OO(3) \times \VV \to \VV,
\end{equation*}
which is linear in $\vv$ and such that
\begin{equation*}
  (g_{1}g_{2}) \star \vv = g_{1} \star (g_{2} \star \vv).
\end{equation*}
Often, $\VV$ is a subspace of $\TT^{n}(\RR^{3})$ and,
\begin{equation*}
  (g \star \bT)(\xx_{1}, \dotsc ,\xx_{n}) := \bT(g^{-1}\xx_{1}, \dotsc ,g^{-1}\xx_{n}), \qquad g \in \OO(3), \, \bT \in \VV,
\end{equation*}
where
\begin{equation*}
  \bT(\xx_{1}, \dotsc ,\xx_{n}) = \sum_{i_{1}, \dotsc , i_{n}} T_{i_{1} \dotsc i_{n}} \, x_{1}^{i_{1}} \dotsm x_{n}^{i_{n}}.
\end{equation*}
Such a representation is then called a \emph{tensorial representation}.

\begin{rem}
  Note that the representations of $\OO(3)$ and $\SO(3)$ on \emph{even-order tensors} are the same because, then,
  \begin{equation*}
    (-\id) \star \bT = \bT, \qquad \forall \, \bT \in \TT^{2n}(\RR^{3}),
  \end{equation*}
  where $\id$ is identity in $\OO(3)$.
\end{rem}

A polynomial function $p$ defined on $\VV$ (\textit{i.e} which can be written as a polynomial in components of $\vv \in \VV$ in any basis) is \emph{invariant} if
\begin{equation*}
  p(g\star \vv) = p(\vv), \qquad \forall g \in \OO(3),\, \forall \vv  \in \VV.
\end{equation*}
The set of $\OO(3)$-invariant polynomial functions is a subalgebra of the polynomial algebra $\RR[\VV]$ of real polynomial functions on $\VV$, which will be denoted by $\RR[\VV]^{\OO(3)}$.

\begin{defn}[Integrity basis]
  A finite set of $\OO(3)$-invariant polynomials $\set{P_{1}, \dotsc , P_{k}}$ over $\VV$ is a \emph{generating set} (also called an \emph{integrity basis}) of the invariant algebra $\RR[\VV]^{\OO(3)}$ if any $\OO(3)$-invariant polynomial $J$ over $\mathbb{V}$ is a polynomial function in $P_{1},\dotsc,P_{k}$, \textit{i.e} if $J$ can be written as
  \begin{equation*}
    J(\vv) = p(P_{1}(\vv), \dotsc ,P_{k}(\vv)), \qquad \vv \in \VV,
  \end{equation*}
  where $p$ is a polynomial function in $k$ variables. An integrity basis is \emph{\emph{minimal}} if no proper subset of it is an integrity basis.
\end{defn}

\begin{ex}[$\VV=\RR^{3}\oplus\dotsb\oplus\RR^{3}$]
  For an $n$-uplet of vectors $(\vv_{1}, \dotsc ,\vv_{n})$, Weyl~\cite{Wey1939} proved that a minimal integrity basis of the diagonal representation of $\OO(3)$ is given by the family
  \begin{equation*}
    \vv_{i} \cdot \vv_{j}, \qquad  i,j = 1, \dotsc , n.
  \end{equation*}
\end{ex}

\begin{ex}[$\VV=\Sym^{2}(\RR^{3})$]
  Another classical example is given by the standard $\OO(3)$-representation on $\Sym^{2}(\RR^{3})$, the space of symmetric second-order tensors on $\RR^{3}$. A minimal integrity basis is given by
  \begin{equation*}
    \tr \ba, \qquad \tr \ba^{2}, \qquad \tr \ba^{3},
  \end{equation*}
  where $\ba\in \Sym^{2}(\RR^{3})$.
\end{ex}

A minimal integrity basis is not unique but its cardinality as well as the degrees of its members are independent of the basis~\cite{Dix1990}. For instance, an alternative minimal integrity basis of $\RR[\Sym^{2}(\RR^{3})]^{\OO(3)}$ is given by the three elementary functions
\begin{equation*}
  \sigma_{1} : = \lambda_{1} + \lambda_{2} + \lambda_{3}, \qquad
  \sigma_{2} : = \lambda_{1}\lambda_{2} + \lambda_{1}\lambda_{3} + \lambda_{2}\lambda_{3}, \qquad
  \sigma_{3} : = \lambda_{1}\lambda_{2}\lambda_{3},
\end{equation*}
where $\lambda_{k}$ are the eigenvalues of the second-order symmetric tensor $\ba$. These two minimal integrity bases are related by invertible polynomial relations, more precisely
\begin{gather*}
  \sigma_{1} = \tr \ba, \qquad \sigma_{2} = \frac{1}{2}\left((\tr \ba)^{2} - \tr \ba^{2}\right),
  \\
  \sigma_{3} = \frac{1}{6}\left((\tr \ba)^{3}-3 \tr \ba \, \tr \ba^{2} + 2\tr \ba^{3}\right),
\end{gather*}
and conversely
\begin{equation*}
  \tr \ba= \sigma_{1},
  \qquad
  \tr \ba^{2}= \sigma_{1}^{2}-2 \sigma_{2},
  \qquad
  \tr \ba^{3} = \sigma_{1}^{3}-3 \sigma_{1} \sigma_{2} +3 \sigma_{3}.
\end{equation*}

For a couple $(\ba, \bb)$ of second-order symmetric tensors, that is for
\begin{equation*}
  \VV=\Sym^{2}(\RR^{3}) \oplus \Sym^{2}(\RR^{3}),
\end{equation*}
a minimal integrity basis for the diagonal $\OO(3)$-representation is known since at least 1958~\cite{SR1958/1959}, and can be found in many references, for instance~\cite{Smi1965,Boe1978,Boe1987,Zhe1994}. More precisely, the following result holds.

\begin{prop}\label{prop:IB-S2+S2}
  The following collection of ten polynomial invariants
  \begin{gather*}
    \tr \ba, \qquad \tr \ba^{2}, \qquad \tr \ba^{3}, \qquad \tr \bb, \qquad \tr \bb^{2}, \qquad \tr \bb^{3},
    \\
    \tr \ba\bb, \qquad \tr \ba^{2}\bb, \qquad \tr \ba \bb^{2}, \qquad \tr \ba^{2}\bb^{2}.
  \end{gather*}
  is a minimal integrity basis for $\RR[\Sym^{2}(\RR^{3}) \oplus \Sym^{2}(\RR^{3})]^{\OO(3)}$.
\end{prop}

For higher order tensors, the determination of an integrity basis is much more complicated and one way to compute such a basis requires first to decompose the tensor space $\VV$ into \emph{irreducible representations} called also an \emph{harmonic decomposition} of $\VV$ (see~\cite{Bac1970,Cow1989,Bae1993,ABR2001,AH2012,OKA2017} for more details). In this decomposition, the irreducible factors are isomorphic to the spaces $\HH^{n}(\RR^{3})$, of $n$th-order \emph{harmonic tensors}. Such a decomposition is, in general, not unique. For the elasticity tensor, for which $E_{ijkl}=E_{ijlk}=E_{klij}$, we can use, for instance, the following explicit decomposition:
\begin{equation}\label{eq:dec-harm-E}
  \bE = (\lambda, \mu, \bd^{\prime}, \bv^{\prime}, \bH),
\end{equation}
with
\begin{equation*}
  \lambda := \tr \bd,  \qquad \mu := \tr \bv, \qquad \bd' := \bd - (\lambda/3) \bq, \qquad \bv' := \bv - (\mu/3) \bq,
\end{equation*}
where $\bd:=  \tr_{12}\bE$ (\textit{i.e.} $d_{ij}=E_{kkij}$) is the \emph{dilatation tensor}, $ \bv:=\tr_{13}\bE$ (\textit{i.e.} $v_{ij}=E_{kikj}$) is the \emph{Voigt tensor} and
\begin{equation*}
  \bH := (\bE^{s})^{\prime} = \bE^{s} - \bq \odot \ba^{\prime} - \frac{7}{30} (\tr \ba)\, \bq \odot \bq, \quad \text{where} \quad \ba := \frac{2}{7}(\bd+2\bv),
\end{equation*}
where $\bq$ is the Euclidean tensor (the scalar product). Note that in any decomposition of the elasticity tensor, the fourth-order component $\bH$ is uniquely defined, which is not the case of the other components.

A minimal integrity basis of $\RR[\HH^{4}]^{\OO(3)}$ was exhibited for the first time by Boehler, Onat and Kirillov~\cite{BKO1994} and republished later by Smith and Bao~\cite{SB1997}. In both cases, the derivation is based on original mathematical results obtained earlier by Shioda~\cite{Shi1967} and von Gall~\cite{vGal1880} on binary forms (see Appendix~\ref{sec:rational-invariants}). The corresponding minimal integrity basis, provided in~\cite{BKO1994}, uses the following second-order covariants, \emph{i.e.} second-order tensor valued functions $\bd(\bH)$, depending of $\bH$ in such a way that
\begin{equation*}\label{eq:cov2}
  g \star \bd(\bH) = \bd(g \star\bH),
\end{equation*}
for all $\bH\in \HH^{4}(\RR^3)$ and $g \in \OO(3)$ (see~\cite{OKDD2018a} for more details).

\begin{thm}[Boehler--Kirillov--Onat]
  \label{prop:IB1-H4}
  Let $\bH \in \HH^{4}(\RR^3)$ and set:
  \begin{equation}\label{eq:boehler-invariants}
    \begin{array} {lll}
      \bd_{2}= \tr_{13} \bH^{2},                 & \bd_{3} = \tr_{13} \bH^{3},              & \bd_{4} = \bd_{2}^{2},                         \\
      \bd_{5} = \bd_{2}(\bH: \bd_{2}),           & \bd_{6} = \bd_{2}^{3},                   & \bd_{7} = \bd_{2}^{2} (\bH: \bd_{2})           \\
      \bd_{8} = \bd_{2}^{2} (\bH^{2} : \bd_{2}), & \bd_{9} = \bd_{2}^{2}(\bH :\bd_{2}^{2}), & \bd_{10} = \bd_{2}^{2} (\bH^{2}: \bd_{2}^{2}).
    \end{array}
  \end{equation}
  A minimal integrity basis for $\bH$ is given by the nine following invariants:
  \begin{equation}\label{eq:IB1-H4}
    J_{k} = \tr \bd_{k} , \qquad k = 2, \dotsc ,10.
  \end{equation}
\end{thm}

\begin{rem}
  The invariant algebra $\RR[\HH^{4}]^{\OO(3)}$ is not free: the polynomials $J_{2}$, \ldots , $J_{7}$ are algebraically independent (\textit{i.e.}, the only polynomial $Q$ such that $Q(J_{2}, \dotsc , J_{7}) = 0$ is the zero polynomial), but $J_{8}$, $J_{9}$, $J_{10}$ are subject to algebraic relations involving the first six invariants and known as \emph{syzygies}, see~\cite{Shi1967,LR2012}.
\end{rem}

Recall that, even if a minimal integrity basis is not unique, its cardinality and the degree of its elements are the same for all bases \cite{Dix1990}. A remarkable observation is that there exists a minimal integrity basis of $\HH^{4}(\RR^3)$, involving only the two second-order covariants $\bd_{2}$, $\bd_{3}$ introduced in~\eqref{eq:boehler-invariants}.

\begin{thm}\label{thm:IB2-H4}
  The following nine polynomial invariants
  \begin{equation} \label{eq:Ikinvariants}
    \begin{array} {lll}
      I_{2}=\tr \bd_{2},                & I_{3}=\tr \bd_{3},      & I_{4}= \tr \bd_{2}^{2},
      \\
      I_{5}=\tr (\bd_{2} \bd_{3}),      & I_{6}= \tr \bd_{2}^{3}, & I_{7}=\tr (\bd_{2}^{2}\bd_{3}),
      \\
      I_{8}=\tr (\bd_{2} \bd_{3}^{2}) , & I_{9}= \tr \bd_{3}^{3}, & I_{10}=\tr (\bd_{2}^{2} \bd_{3}^{2}).
    \end{array}
  \end{equation}
  form a minimal integrity basis of $\RR[\HH^{4}]^{\OO(3)}$.
\end{thm}

The proof follows from the fact that there are algebraic relations between the two sets of invariants provided below.
\begin{equation*}
  I_{2} = J_{2}, \qquad I_{3} = J_{3}, \qquad I_{4} = J_{4}, \qquad I_{6} = J_{6}.
\end{equation*}
Then, we have
\begin{align*}
  I_{5}  & = \frac{1}{6}\big(3J_{5} + 2J_{2} J_{3}\big),                                                                                \\
  I_{7}  & = \frac{1}{6}\big(3J_{7} + 2J_{4} J_{3}\big),                                                                                \\
  I_{8}  & = \frac{1}{1620}\big(1080J_{8} - 1230J_{6}J_{2} + 495J_{5}J_{3} - 216J_{4}^{2} + 1197J_{4}J_{2}^{2}
  \\
         & \quad + 140J_{3}^{2}J_{2} - 237J_{2}^{4}\big),
  \\
  I_{9}  & = \frac{1}{19440}\big(5184 J_{9} - 6480 J_{7}J_{2} + 9456 J_{6}J_{3} + 2025{5}J_{2}^{2}
  \\
         & \quad  - 7974 J_{4}J_{3}J_{2} + 2500 J_{3}^{3} + 1596 J_{3}J_{2}^{3} \big),
  \\
  I_{10} & = \frac{1}{1630}\big(1080 J_{10} - 675 J_{8}J_{2} + 495 J_{7}J_{3} + 24 J_{6}J_{4} - 117 J_{6}J_{2}^{2} - 171 J_{4}^{2}J_{2}
  \\
         & \quad + 190 J_{4}J_{3}^{2} + 228 J_{4}J_{2}^{3} - 45 J_{2}^{5}\big),
\end{align*}
and conversely
\begin{align*}
  J_{5}  & = \frac{1}{3} \big( 6I_{5} - 2I_{2}I_{3} \big),                                                                                                                   \\
  J_{7}  & = \frac{1}{3} \big( 6I_{7} - 2I_{4}I_{3} \big),                                                                                                                   \\
  J_{8}  & = \frac{1}{2160} \big( 3240I_{8} - 1980I_{5}I_{3} + 2460I_{6}I_{2} + 380I_{3}^{2} I_{2} + 432I_{4}^{2} - 2394I_{4}I_{2}^{2}
  \\
         & \quad + 474I_{2}^{4} \big),
  \\
  J_{9}  & = \frac{1}{10368} \big( 38880I_{9} + 25920I_{7}I_{2} - 8100I_{5}I_{2}^{2} - 5000I_{3}^{3} - 18912I_{3}I_{6}
  \\
         & \quad + 7308I_{3}I_{4}I_{2} - 492I_{3}I_{2}^{3} \big),
  \\
  J_{10} & = \frac{1}{17280} \big( 25920I_{10} + 16200I_{8}I_{2} - 15840I_{7}I_{3} - 9900 I_{5}I_{3}I_{2} + 2240I_{3}^{2}I_{4}
  \\
         & \quad + 1900 I_{3}^{2}I_{2}^{2} - 384I_{6}I_{4}+14172I_{6}I_{2}^{2} + 4896I_{4}^{2}I_{2} - 15618 I_{4}I_{2}^{3} + 3090 I_{2}^{5} \big).
\end{align*}

% ----------------------------------------------------------------
\section{Separating sets}
\label{sec:separating-sets}
% ----------------------------------------------------------------

The weaker concept of \emph{separating set}, often called a \emph{functional basis} in the mechanical community \cite{WP1964,Boe1987a} (see \cite{Duf2008,Kem2009,DK2015} for alternative definitions in the mathematical community), is formulated in invariant theory as follows.

\begin{defn}[Separating set]\label{Def:funcbas}
  A finite set of $\OO(3)$-invariant functions $\set{s_{1}, \dotsc , s_{n}}$ over $\VV$ is a \emph{separating set} if
  \begin{equation*}
    s_{i}(\vv_{1}) = s_{i}(\vv_{2}), \quad i = 1, \dotsc , n  \implies \exists g \in \OO(3),\, \vv_{1} = g \star \vv_{2}.
  \end{equation*}
  for all $\vv_{1}, \vv_{2} \in \VV$. A separating set is minimal if no proper subset of it is a separating set.
\end{defn}

Note that this definition is very general and the functions $s_{1}$, \ldots, $s_{n}$ are not required to be polynomial.

\begin{rem}\label{rem:interity-separting}
  A remarkable fact is that an integrity basis of $\RR[\VV]^{\OO(3)}$, the algebra of real $\OO(3)$-invariant polynomials over $\VV$, is also a \emph{separating set}~\cite[Appendix C]{AS1983}. However, the cardinal of an integrity basis can be very big (for instance, it is of 297 for the 3D elasticity tensor~\cite{OKA2017}). But, even if no general result exists, the cardinal of a polynomial separating set can be smaller than the cardinal of a minimal integrity basis (see for instance~\cite{Zhe1994}).
\end{rem}

An even weaker concept was suggested in~\cite{BKO1994}, but requires first to define what is meant by \emph{generic tensors} (also called \emph{tensors in general position}). This can be done rigorously by introducing the \emph{Zariski topology} on $\VV$, which is defined by specifying its \emph{closed sets} rather than its \emph{open sets} (see~\cite{Har1977} for more details). A closed set in the Zariski topology is defined as
\begin{equation*}
  Z := \set{\vv \in \VV; \; f(\vv) = 0, \, \forall f \in S}
\end{equation*}
where $S$ is any set of polynomials in $\vv$.

\begin{rem}
  A remarkable fact concerning this topology is that a non-empty closed set is either the whole space or has Lebesgue measure zero~\cite{CT2005,PS2016}.
\end{rem}

A Zariski open set is defined as the complementary set $Z^{c}$ of a closed Zariski set. A non-empty Zariski open set is moreover \emph{open and dense} in the usual topology.

\begin{ex}
  On $\VV = \RR^{3}\oplus\RR^{3}\oplus\RR^{3}$, the following set
  \begin{equation*}
    Z := \set{(\vv_{1}, \vv_{2}, \vv_{3}) \in \VV;\; \det(\vv_{1}, \vv_{2}, \vv_{3}) = 0}
  \end{equation*}
  is a Zariski closed set and
  \begin{equation*}
    Z^{c} = \set{(\vv_{1}, \vv_{2}, \vv_{3}) \in \VV;\; \det(\vv_{1}, \vv_{2}, \vv_{3}) \ne 0}
  \end{equation*}
  is a Zariski open set.
\end{ex}

\begin{defn}[Genericity]\label{def:genericity}
  A vector $\vv$ belonging to some (finite dimensional) vector space $\VV$ is called \emph{generic} (or as \emph{in general position} by algebraic geometers) if it belongs to a non-empty Zariski open set of $\VV$.
\end{defn}

Coming back to our definition of generic tensors, this means that informally speaking, the probability of a randomly chosen tensor being generic is 1 and that we omit, in the results, some tensors satisfying certain algebraic relations. Note, however, that this notion of genericity is arbitrary and there is a lot of freedom in the choice of \emph{such a class} of generic tensors.

\begin{defn}[Weak separating set]
  Given some non-empty Zariski open set $Z^{c} \subset \VV$, a finite set of $\OO(3)$-invariant functions $\set{s_{1}, \dotsc , s_{n}}$ over $\VV$ is called a \emph{weak separating set} (or a \emph{weak functional basis}) if
  \begin{equation*}
    s_{i}(\vv_{1}) = s_{i}(\vv_{2}), \quad i = 1, \dotsc , n  \implies \exists g \in \OO(3),\, \vv_{1} = g \star \vv_{2}.
  \end{equation*}
  for all $\vv_{1}, \vv_{2} \in Z^{c}$.
\end{defn}

The notion of \emph{minimal cardinality} for weak functional bases can also be formulated in a \emph{given class of functions}. We shall say that a weak functional basis is minimal if their is no other weak functional basis with smaller cardinal in the same class of functions. If some results exist for the class of polynomial functions in complex algebraic geometry~\cite{Duf2008}, where some bounds on the cardinal of a minimal weak functional basis are provided, it is not totaly clear how they can be directly translated into the realm of real algebraic geometry.

Besides (weak) functional bases of polynomial functions, there are also results on functional bases of rational functions~\cite{HK2007,GHP2018} (which are necessarily \emph{weak} since tensors for which the denominators vanish are forbidden). For instance, Maeda~\cite{Mae1990} provided a separating set of 6 rational invariants for binary octavics (complex polynomials homogeneous of degree 8 in two variables), which are closely related to harmonic tensors of order 4 (see Appendix~\ref{sec:rational-invariants}). Using this result, we provide in Appendix~\ref{sec:Maeda-invariants} a separating set of 6 rational invariants for $\HH^{4}(\RR^3)$.
This set is minimal because one cannot produce a set of separating invariants of cardinality lower than the \emph{transcendence degree}, which is the maximal number of algebraic independent elements in the fractional field of the invariant algebra~\cite[Page 26]{Bri1996}. For the elasticity tensor, this minimal number is
\begin{equation*}
  \dim \Ela - \dim \OO(3) = 18.
\end{equation*}
On this matter, there is a paper by Ostrosablin~\cite{Ost1998} who suggests a system of 18 separating rational invariants, but no rigourous proof of this result seems to be available in the literature.

Finally, there is a third notion of separability which should not be confused with the preceding ones. It concerns \emph{local separability} and can be formulated as follows.

\begin{defn}[Locally separating set]
  A finite set of $\OO(3)$-invariant functions $\set{s_{1}, \dotsc , s_{n}}$ over $\VV$ is \emph{locally separating in the neighbourhood $U \subset \VV$ of $\vv_{0}$} (for the usual topology of $\VV$) if and only if
  \begin{equation*}
    s_{i}(\vv_{1}) = s_{i}(\vv_{2}), \quad i = 1, \dotsc , n  \implies \exists g \in \OO(3),\, \vv_{1} = g \star \vv_{2}.
  \end{equation*}
  for all $\vv_{1}, \vv_{2} \in U$.
\end{defn}

Such a set can be considered as a ``local chart'' (\textit{i.e} local coordinates) around the orbit of $\vv_{0}$ in the orbit space $\VV/\OO(3)$ (which is not a smooth manifold anyway). A locally separating set of 18 invariants (but not polynomial) for elasticity tensors which have 6 distinct Kelvin moduli~\cite{BBS2007} has been produced in~\cite{BBS2008}.

\begin{rem}
  Since an integrity basis $J = (J_{1}, \dotsc , J_{297})$ is known for the elasticity tensor, one can find a locally separating set of 18 invariants (\textit{i.e.} the minimal number) around each tensor $\bE^{0}$ for which the Jacobian matrix
  \begin{equation*}
    dJ = \left(\frac{\partial J_{p}}{\partial E_{ijkl}}\right)
  \end{equation*}
  has maximal rank 18. Indeed, one can extract from $dJ$, a submatrix
  \begin{equation*}
    (dJ_{p_{1}}, \dotsc, dJ_{p_{18}})
  \end{equation*}
  of rank 18, construct a local \emph{cross-section} as in~\cite[Page 161]{Olv1999} and show that $J_{p_{1}}, \dotsc , J_{p_{18}}$ are locally separating around $\bE^{0}$.
\end{rem}

These several notions of separability differ by the size of the subset $U$ of $\VV$, on which the separating property is defined. The strongest one is the first one (separating set) because the separating property is global and defined over the whole vector space $\VV$. In particular, the minimal integrity basis --- of 297 invariants --- for elasticity tensors produced in~\cite{OKA2017} is a global, albeit non minimal, separating set over the full vector space $\Ela$. A Zariski's open sets $Z^{c}$ being very large (open and dense in the usual topology of $\VV$), the second notion (weak separating set) separates most orbits (except a few ones which constitute a set of zero Lebesgue measure over $\VV$). The last one (local separating set) is the weaker, it separates only tensors in a given neighbourhood $U$ of a given point $\vv_{0} \in \VV$.

% ----------------------------------------------------------------
\section{Weak separating sets for elasticity tensors}
\label{sec:seps-elasticity}
% ----------------------------------------------------------------

In~\cite{BKO1994}, Boehler, Kirillov and Onat introduced a set of \emph{generic elasticity tensors} and provided, for this set, a weak separating set of 39 polynomial invariants. Their generic tensors are defined as those for which the second-order covariants $\bd_{2}$ and $\bd_{3}$ do not share a common principal axis. This is equivalent to say that the symmetry class of the pair $(\bd_{2},\bd_{3})$ is triclinic (its symmetry group is reduced to the identity). This condition defines a Zariski open set $Z^{c}$. Some polynomial equations defining the complementary set $Z$ were detailed by the authors using tensor's components. An intrinsic and covariant formulation of these conditions can be formulated as follows (see~\cite[Theorem 8.5]{OKDD2018a})
\begin{equation*}
  (\bd_{2}\vv_{5}) \times \vv_{5} \ne 0 \quad \text{or} \quad (\bd_{3}\vv_{5}) \times \vv_{5} \ne 0,
\end{equation*}
where $\vv_{5} := \beps \2dots (\bd_{2}\bd_{3}-\bd_{3}\bd_{2})$ is a first-order covariant of $\bH$. In the present work, we shall consider a smaller Zariski open set by restricting to elasticity tensors for which \emph{$\bd_{2}$ is furthermore orthotropic} (three distinct eigenvalues). This is equivalent to add the polynomial condition ${\bd_{2}}^{2} \times \bd_{2} \ne 0$ (see~\cite[Lemma 8.1]{OKDD2018a}).

\begin{rem}
  Note that, if the pair $(\bd_{2},\bd_{3})$ is triclinic, then $\bH$ (and hence $\bE$) is triclinic, since a tensor cannot be less symmetric than its covariants. However, the converse does not hold: it is not true that for any triclinic elasticity tensor $\bE$, the pair of second-order covariants $(\bd_{2},\bd_{3})$ is triclinic, the later condition is stronger.
\end{rem}

We will now formulate our main theorem, using the following notations: $(\ba \bb)^{s} := \ba \bb+ \bb \ba$ is the symmetrized matrix product, and $[\ba, \bb] := \ba \bb -\bb \ba$ is the commutator of two second-order symmetric tensors $\ba$, $\bb$.

\begin{thm}\label{thm:main}
  Let $\bE = (\lambda, \mu, \bd^{\prime}, \bv^{\prime}, \bH)$ be an elasticity tensor, $\bd_{2}=\tr_{13}\bH^2$ and $\bd_3=\tr_{13}\bH^3$. Then, the following 21 polynomial invariants, $\lambda=\tr \bd$, $\mu=\tr \bv$,
  \begin{gather*}
    I_{2} := \tr \bd_{2}, \quad I_{3} := \tr \bd_{3}, \quad I_{4} := \tr \bd_{2}^{2}, \quad I_{5} := \tr (\bd_{2} \bd_{3}), \quad I_{6} := \tr \bd_{2}^{3},
    \\
    I_{7} := \tr (\bd_{2}^{2}\bd_{3}), \quad I_{8} := \tr (\bd_{2} \bd_{3}^{2}), \quad I_{9} := \tr \bd_{3}^{3}, \quad I_{10} := \tr (\bd_{2}^{2} \bd_{3}^{2}),
    \\
    D_{3} := \bd^{\prime} \2dots \bd_{2}, \quad  D_{4} := \bd^{\prime} \2dots \bd_{3}, \quad D_{5} := \bd^{\prime} \2dots \bd_{2}^{2},
    \\
    D_{6} := \bd^{\prime} \2dots (\bd_{2}\bd_{3})^{s}, \quad D_{11} := \bd^{\prime} \2dots [\bd_{2},\bd_{3}]^{2},
    \\
    V_{3} := \bv^{\prime} \2dots \bd_{2}, \quad V_{4} := \bv^{\prime} \2dots \bd_{3}, \quad V_{5} := \bv^{\prime} \2dots \bd_{2}^{2 },
    \\
    V_{6} := \bv^{\prime} \2dots (\bd_{2}\bd_{3})^{s}, \quad V_{11} := \bv^{\prime} \2dots [\bd_{2}, \bd_{3}]^{2},
  \end{gather*}
  separate generic tensors $\bE$, satisfying the following conditions: (1) the pair $(\bd_{2},\bd_{3})$ is triclinic, and (2) $\bd_{2}$ is orthotropic.
\end{thm}

\begin{rem}
  Note that if condition (1) is satisfied, then, either $\bd_{2}$ or $\bd_{3}$ is orthotropic by Lemma~\ref{lem:basis-ab}. Thus, one could omit condition (2) (as in \cite{BKO1994}) and formulate a new separating result on this larger Zariski open set. However, the price to pay is to add the two invariants $\bd^{\prime}:\bd_{3}^{2}$ and $\bv^{\prime}:\bd_{3}^{2}$ to the list in Theorem~\ref{thm:main}, increasing its cardinal from 21 to 23 (but still below the 39 invariants of~\cite{BKO1994}). Indeed, $\bd_{3}$ can play the role of $\bd_{2}$ in the proof of Theorem~\ref{thm:main}, in that case.
\end{rem}

The proof of Theorem~\ref{thm:main} is based on the following lemma.

\begin{lem}\label{lem:basis-ab}
  Let $(\ba, \bb)$ be a triclinic pair of symmetric second-order tensors. Then at least one of them is orthotropic, say $\ba$, and in that case
  \begin{equation*}
    \mathcal{B} = \left(\bq, \ba, \bb, \ba^{2}, (\ba \bb)^{s}, [\ba, \bb]^{2}\right)
  \end{equation*}
  is a basis of $\Sym^{2}(\RR^{3})$, the space of symmetric second-order tensors.
\end{lem}

\begin{proof}
  Note first that $\ba$ and $\bb$ cannot be both transversely isotropic (\textit{i.e.} having both only two different eigenvalues), otherwise the pair $(\ba,\bb)$ would have necessarily a common eigenvector and would be not be triclinic. Suppose thus that $\ba$ is orthotropic. Without loss of generality, we can assume that $\ba = \mathrm{diag}(\lambda_{1}, \lambda_{2},\lambda_{3})$ is diagonal with $\lambda_i\neq \lambda_j$ for $i\neq j$. But then, $(\bq, \ba, \ba^{2})$ is a basis of the space of diagonal matrices, noted $\mathrm{Diag}$, and therefore $\mathcal{B}$ contains $\be_{11}, \be_{22}, \be_{33}$ where
  \begin{equation*}
    \be_{ij} =
    \begin{cases}
      \eee_{i}\otimes\eee_{i},\quad \text{if}\ i=j                             \\
      \eee_{i}\otimes\eee_{j}+\eee_{j}\otimes\eee_{i},\quad \text{if}\ i\neq j \\
    \end{cases}
  \end{equation*}
  We will now show that $\mathcal{B}$ contains also $\be_{12}, \be_{13}, \be_{23}$. Let's write
  \begin{equation*}
    \bb = x\, \be_{23} + y\, \be_{13} + z\, \be_{12}, \mod \mathrm{Diag}.
  \end{equation*}
  where modulo $\mathrm{Diag}$ means that the equality holds up to a diagonal matrix that we don't need to precise.
  We cannot have $(x, y)=(0,0)$, nor $(x, z)=(0,0)$, nor $(y, z)=(0,0)$, otherwise $\ba$ and $\bb$ would share a common eigenvector and would not be triclinic. We have then
  \begin{equation*}
    (\ba \bb)^{s} = (\lambda_{2} + \lambda_{3})x \, \be_{23} + (\lambda_{1} + \lambda_{3})y \, \be_{13} + (\lambda_{1} + \lambda_{2})z \, \be_{12}, \mod \mathrm{Diag}.
  \end{equation*}
  and
  \begin{multline*}
    [\ba, \bb]^{2} = ((\lambda_{1} - \lambda_{2})\lambda_{3} + \lambda_{1}\lambda_{2} - \lambda_{1}^{2})yz \, \be_{23}
    +  ((\lambda_{2} - \lambda_{1}) \lambda_{3} - \lambda_{2}^{2} + \lambda_{1} \lambda_{2})xz \, \be_{13} \\
    + (- \lambda_{3}^{2} + (\lambda_{2} + \lambda_{1} )\lambda_{3} - \lambda_{1} \lambda_{2})xy \, \be_{12}, \mod \mathrm{Diag}.
  \end{multline*}
  The question is then reduced to check whether $\bb$, $(\ba \bb)^{s}$ and $[\ba, \bb]^{2}$ are linearly independent modulo $\mathrm{Diag}$. To do so, we calculate the determinant of the matrix
  \begin{equation*}
    M =
    \begin{pmatrix}
      b_{23} & {(\ba \bb)^{s}}_{23} & {[\ba, \bb]^{2}}_{23} \\
      b_{13} & {(\ba \bb)^{s}}_{13} & {[\ba, \bb]^{2}}_{13} \\
      b_{12} & {(\ba \bb)^{s}}_{12} & {[\ba, \bb]^{2}}_{12} \\
    \end{pmatrix}
  \end{equation*}
  and find
  \begin{equation*}
    \det M = (\lambda_{2}-\lambda_{1}) (\lambda_{3}-\lambda_{1}) (\lambda_{3}-\lambda_{2}) \left(x^{2}y^{2} + y^{2}z^{2} + z^{2}x^{2}\right),
  \end{equation*}
  which does not vanish since $\ba$ is orthotropic. This achieves the proof.
\end{proof}

\begin{proof}[Proof of Theorem~\ref{thm:main}]
  Let $\bE = (\lambda, \mu, \bd^{\prime}, \bv^{\prime}, \bH)$ be an elasticity tensor satisfying the conditions (1) and (2) of Theorem~\ref{thm:main}. Then, by Lemma~\ref{lem:basis-ab},
  \begin{equation*}
    \mathcal{B} = \left(\bq, \bd_{2}, \bd_{3}, {\bd_{2}}^{2}, (\bd_{2}\bd_{3})^{s}, [\bd_{2}, \bd_{3}]^{2}\right)
  \end{equation*}
  is a basis of $\Sym^{2}(\RR^{3})$. Thus, if we set
  \begin{equation*}
    \pmb\epsilon_{1} := \bd_{2}, \quad \pmb\epsilon_{2} := \bd_{3}, \quad \pmb\epsilon_{3} := {\bd_{2}}^{2}, \quad \pmb\epsilon_{4} := (\bd_{2}\bd_{3})^{s}, \quad \pmb\epsilon_{5} := [\bd_{2}, \bd_{3}]^{2},
  \end{equation*}
  and define $\pmb\epsilon^{\prime}$ as the deviatoric part of $\pmb\epsilon$, then, $\mathcal{B^{\prime}} = (\pmb\epsilon^{\prime}_{\alpha})$ is a basis of the 5-dimensional vector space $\HH^{2}(\RR^{3})$, \textit{i.e.} of the space of deviatoric second-order tensors. In particular, the second-order harmonic components $(\bd^{\prime}, \bv^{\prime})$ of $\bE$ can be expressed in this basis as
  \begin{equation*}
    \bd^{\prime}  = \sum_{\alpha=1}^{5} d^{\prime}_{\alpha}\pmb\epsilon^{\prime}_{\alpha}, \qquad \bv^{\prime} = \sum_{\alpha=1}^{5} v^{\prime}_{\alpha}\pmb\epsilon^{\prime}_{\alpha}.
  \end{equation*}
  We will now show that the components $d^{\prime}_{\alpha}$ and $v^{\prime}_{\alpha}$ are rational expressions of the polynomial invariants $I_{k}$, $D_{k}$ and $V_{k}$ introduced in Theorem~\ref{thm:main}. To do so, we shall introduce the \emph{Gram matrix} $G = (G_{\alpha\beta})$, where
  \begin{equation*}
    G_{\alpha\beta} = \pmb\epsilon^{\prime}_{\alpha} \2dots \pmb\epsilon^{\prime}_{\beta}
  \end{equation*}
  are the components of the canonical scalar product on $\HH^{2}(\RR^{3})$ in this basis. Note that $G$ is positive definite and that its components are polynomial invariants of $\bH$. They can thus be expressed as polynomial functions of the invariants $I_{2}, \dotsc, I_{10}$, which form an integrity basis of $\RR[\HH^{4}]^{\OO(3)}$. Now, we have
  \begin{equation*}
    \bd^{\prime} \2dots \pmb\epsilon^{\prime}_{\beta} = \sum_{\alpha=1}^{5} d^{\prime}_{\alpha} G_{\alpha\beta}, \qquad \bv^{\prime} \2dots \pmb\epsilon^{\prime}_{\beta} = \sum_{\alpha=1}^{5} v^{\prime}_{\alpha} G_{\alpha\beta},
  \end{equation*}
  and since
  \begin{equation*}
    \bd^{\prime}:\pmb\epsilon^{\prime} = \bd^{\prime}:\pmb\epsilon, \quad \text{and} \quad \bv^{\prime}:\pmb\epsilon^{\prime} = \bv^{\prime}:\pmb\epsilon,
  \end{equation*}
  we get
  \begin{equation*}
    (D_{3} \; D_{4} \; D_{5} \; D_{6} \; D_{11}) = (d^{\prime}_{1} \; d^{\prime}_{2} \; d^{\prime}_{3} \; d^{\prime}_{4} \; d^{\prime}_{5})G,
  \end{equation*}
  and
  \begin{equation*}
    (V_{3} \; V_{4} \; V_{5} \; V_{6} \; V_{11}) = (v^{\prime}_{1} \; v^{\prime}_{2} \; v^{\prime}_{3} \; v^{\prime}_{4} \; v^{\prime}_{5})G.
  \end{equation*}

  Inverting these linear systems, we deduce that $d^{\prime}_{\alpha}$ and $v^{\prime}_{\alpha}$ are rational expressions of $I_{k}$, $D_{k}$ and $V_{k}$, where the common denominator $\det G$ depends only on the $I_{k}$. Consider now two generic elasticity tensors
  \begin{equation*}
    \bE = (\lambda, \mu, \bd^{\prime}, \bv^{\prime}, \bH), \quad \text{and } \quad \overline{\bE} = (\overline{\lambda}, \overline{\mu}, \overline{\bd}^{\prime}, \overline{\bv}^{\prime}, \overline{\bH})
  \end{equation*}
  for which the 21 invariants defined in Theorem~\ref{thm:main} are the same. Then, by Theorem~\ref{thm:IB2-H4} and Remark~\ref{rem:interity-separting}, there exists $g\in\OO(3)$ such that
  \begin{equation*}
    \overline{\bH} = g \star \bH .
  \end{equation*}
  We get thus
  \begin{equation*}
    \overline{\bd}_{2} = g \star \bd_{2}, \qquad \overline{\bd}_{3} = g \star \bd_{3}.
  \end{equation*}
  Hence the two bases of $\Sym^{2}(\RR^{3})$, $(\pmb\epsilon^{\prime}_{\alpha}(\bH))$ and $(\pmb\epsilon^{\prime}_{\alpha}(\overline{\bH}))$ are related by $g$
  \begin{equation*}
    \pmb\epsilon^{\prime}_{\alpha}(\overline{\bH}) = g \star \pmb\epsilon^{\prime}_{\alpha}(\bH),
  \end{equation*}
  and the corresponding Gram matrices are equal, $\overline{G}=G$. Moreover, the components of $\bd^{\prime}$, $\bv^{\prime}$ in $(\pmb\epsilon^{\prime}_{\alpha}(\bH))$ and the components of $\overline{\bd}^{\prime}$, $\overline{\bv}^{\prime}$ in $(\pmb\epsilon^{\prime}_{\alpha}(\overline{\bH}))$ are the same (since the invariants $D_{k}$ and $V_{k}$ have the same value on both tensors). Therefore, we have
  \begin{equation*}
    \overline{\bd}^{\prime} = g \star \bd^{\prime}, \qquad \overline{\bv}^{\prime} = g \star \bv^{\prime}.
  \end{equation*}
  Finally, since $\overline{\lambda} = \lambda$ and $\overline{\mu} = \mu$, we get
  \begin{equation*}
    \overline{\bE} = (\overline{\bH}, \overline{\bd}^{\prime}, \overline{\bv}^{\prime}, \overline{\lambda}, \overline{\mu}) = (g \star \bH, g \star \bd^{\prime}, g \star \bv^{\prime}, \lambda, \mu) = g \star \bE,
  \end{equation*}
  which achieves the proof.
\end{proof}

Note that in the proof of Theorem~\ref{thm:main}, the nine invariants $I_{k}$ were only used to separate the fourth-order harmonic tensors $\bH$ and $\overline{\bH}$. Thus these nine invariants can be substituted by any other separating set for $\HH^{4}(\RR^3)$ without changing the final result. In Appendix~\ref{sec:Maeda-invariants}, we provide a set of 6 separating rational invariants for $\HH^{4}(\RR^3)$
\begin{equation*}
  i_{2}, \quad i_{3}, \quad i_{4}, \quad  k_{4},  \quad k_{8},  \quad k_{9},
\end{equation*}
obtained by translating the 6 generators of the field of rational invariants of the binary octavic calculated by Maeda in~\cite{Mae1990}. We get therefore the following first corollary.

\begin{cor}\label{cor:main-18}
  The following 18 rational invariants
  \begin{gather*}
    \lambda, \quad \mu , \quad i_{2}, \quad i_{3}, \quad i_{4}, \quad k_{4},  \quad k_{8},  \quad k_{9},
    \\
    D_{3}, \quad D_{4}, \quad D_{5}, \quad D_{6}, \quad D_{11}, \quad V_{3}, \quad V_{4}, \quad V_{5}, \quad V_{6}, \quad V_{11}
  \end{gather*}
  separate generic tensors $\bE = (\lambda, \mu, \bd^{\prime}, \bv^{\prime}, \bH)$, satisfying the following conditions: (1) the pair $(\bd_{2},\bd_{3})$ is triclinic, and (2) $\bd_{2}$ is orthotropic.
\end{cor}

In Theorem~\ref{thm:Maeda-H4}, it can be observed that the denominator of each rational invariant
\begin{equation*}
  i_{2}, \quad i_{3}, \quad i_{4}, \quad k_{4},  \quad k_{8},  \quad k_{9},
\end{equation*}
is a power of the polynomial invariant of degree 12
\begin{equation*}
  M_{12} := \norm{{\bd_{2}}^{2} \times \bd_{2}}^{2}.
\end{equation*}
where the generalized cross product $\times$ was defined in~\eqref{eq:cross-product}. Besides, it was shown in~\cite[Lemma 8.1]{OKDD2018a} that ${\bd_{2}}^{2} \times \bd_{2} \ne 0$ if and only if $\bd_{2}$ is orthotropic. We have thus the following second corollary.

\begin{cor}\label{cor:main-19}
  The following 19 polynomial invariants
  \begin{gather*}
    \lambda, \qquad \mu, \qquad M_{12}
    \\
    {K_{14}} := M_{12}\, i_{2}, \qquad  K_{27} := {M_{12}}^{2}\, i_{3},  \qquad  K_{40i} := {M_{12}}^{3} \, i_{4},
    \\
    K_{40k} := {M_{12}}^{3}\, k_{4}, \qquad K_{80} := {M_{12}}^{6}\, k_{8},  \qquad  K_{93} := {M_{12}}^{7}\, k_{9},
    \\
    D_{3}, \quad D_{4}, \quad D_{5}, \quad D_{6}, \quad D_{11}, \quad  V_{3}, \quad V_{4}, \quad V_{5}, \quad V_{6}, \quad  V_{11},
  \end{gather*}
  separate generic tensors {$\bE = (\lambda, \mu, \bd^{\prime}, \bv^{\prime}, \bH)$}, satisfying the following conditions: (1) the pair $(\bd_{2},\bd_{3})$ is triclinic, and (2) $\bd_{2}$ is orthotropic.
\end{cor}

% ----------------------------------------------------------------
\appendix
% ----------------------------------------------------------------

% ----------------------------------------------------------------
\section{Rational invariants}
\label{sec:rational-invariants}
% ----------------------------------------------------------------

In this appendix, we detail the link between polynomial and rational invariants of $\HH^{n}(\CC^{3})$ and the space of binary forms $\Sn{2n}$. Recall that a binary form $\ff$ of degree $k$ is a homogeneous complex polynomial in two variables $u,v$ of degree $k$:
\begin{equation*}
  \ff(\bxi) = a_{0}u^{k} + a_{1}u^{k-1}v + \dotsb + a_{k-1}uv^{k-1} + a_{k}v^{k},
\end{equation*}
where $\bxi = (u,v)$ and $a_{i}\in \CC$. The set of all binary forms of degree $k$, noted $\Sn{k}$, is a complex vector space of dimension $k + 1$. The special linear group
\begin{equation*}
  \SL(2,\CC) : =  \set{\gamma: =
    \begin{pmatrix}
      a & b \\
      c & d
    \end{pmatrix}
    ,\quad ad-bc = 1}
\end{equation*}
acts naturally on $\CC^{2}$ and induces a left action on $\Sn{k}$, given by
\begin{equation*}
  (\gamma \star \ff)(\bxi): = \ff(\gamma^{-1} \bxi),
\end{equation*}
where $\gamma\in \SL(2,\CC)$.

Binary forms of degree $2n$ are closely related to harmonic tensors of degree $n$ (we refer to~\cite{OKA2017,OKDD2018a} for more details) in the following way. Every totally symmetric tensor $\bS$ of order $n$ defines an homogeneous polynomial of degree $n$
\begin{equation*}
  \rp(\xx) = \bS(\xx, \dotsc, \xx)
\end{equation*}
which can be seen to be an isomorphism. In this correspondence, harmonic tensors (with vanishing traces) correspond to harmonic polynomials (with vanishing Laplacian). Now, there is an equivariant isomorphism between the space $\Hn{n}(\CC^{3})$ of complex harmonic polynomials of degree $n$ and binary forms of degree $2n$. This isomorphism is induced by the \emph{Cartan map}
\begin{equation}\label{eq:Cartan-map}
  \phi : \CC^{2} \to \CC^{3}, \qquad (u,v) \mapsto \left( \frac{u^{2} + v^{2}}{2}, \frac{u^{2} - v^{2}}{2i}, iuv \right),
\end{equation}
and is given by
\begin{equation*}
  \phi^{*} : \Hn{n}(\CC^{3}) \to \Sn{2n}, \qquad \rh \mapsto \rh \circ \phi .
\end{equation*}
This isomorphism is moreover $\SL(2,\CC)$-equivariant. Indeed, the adjoint representation $\Ad$ of $\SL(2, \CC)$ on its Lie algebra $\slc(2, \CC)$ (which is isomorphic to $\CC^{3}$), preserves the quadratic form $\det m$, where $m \in \slc(2, \CC)$, and induces a group morphism from $\SL(2, \CC)$ to
\begin{equation*}
  \SO(3,\CC) : =  \set{P \in \mathrm{M}_{3}(\CC); \; P^{t}P = \id ,\, \det P = 1}.
\end{equation*}
The isomorphism $\phi^{*}$ between $\Hn{n}(\CC^{3})$ and $\Sn{2n}$ is thus equivariant in the following sense:
\begin{equation*}
  \phi^{*}( \Ad_{\gamma} \star \rh) = \gamma \star \phi^{*}(\rh), \qquad \rh \in \Hn{n}(\CC^{3}), \, \gamma \in \SL(2, \CC),
\end{equation*}
and the invariant algebras $\CC[\Hn{n}(\CC^{3})]^{\SO(3,\CC)}$ and $\CC[\Sn{2n}]^{\SL(2, \CC)}$ are isomorphic.

\begin{defn}
  The \emph{transvectant} of index $r$ of two binary forms $\ff\in \Sn{p}$ and $\bg\in \Sn{q}$ is defined as
  \begin{equation}\label{eq:Transv}
    \trans{\ff}{\bg}{r} = \frac{(p-r)!(q-r)!}{p! q!} \sum_{i = 0}^{r}(-1)^{i} \binom{r}{i} \frac{\partial^{r} \ff}{\partial u^{r-i} \partial v^{i}} \frac{\partial^{r} \bg}{\partial u^{i} \partial v^{r-i}},
  \end{equation}
  which is a binary form of degree $p + q - 2r$ (which vanishes if $r > \min(p,q)$).
\end{defn}

The invariant algebra of $\Sn{n}$ is generated by iterated \emph{transvectants}~\cite{Olv1999}. The tensorial operations between \emph{totally symmetric tensors}, introduced in the notations section, allow to traduce these transvectants into tensorial operations. Each of them has a polynomial counterpart (see~\cite{OKDD2018a}), which we detail below. In what follows, totally symmetric tensors $\bS^{1}, \bS^{2}$, of respective order $n_{1}$, $n_{2}$, correspond to the polynomials $\rp_{1}, \rp_{2}$, of respective degree $n_{1}$, $n_{2}$.

\begin{itemize}
  \item The \emph{symmetric tensor product}~\eqref{eq:symmetric-tensor-product} $\bS^{1} \odot \bS^{2}$ corresponds to the standard product of polynomials
        \begin{equation*}
          \rp = \rp_{1}\, \rp_{2}.
        \end{equation*}
  \item The \emph{symmetric $r$-contraction}~\eqref{eq:symmetric-r-contraction} $\bS^{1} \symrdots{r} \bS^{2}$ corresponds to the polynomial
        \begin{equation*}
          \rp = \frac{(n_{1}-r)!}{n_{1}!}\frac{(n_{2}-r)!}{n_{2}!} \sum_{k_{1}+k_{2}+k_{3}=r} \frac{r!}{k_{1}!k_{2}!k_{3}!}\frac{\partial^r \rp_{1}}{\partial x^{k_{1}}\partial y^{k_{2}}\partial z^{k_{3}}}\frac{\partial^r \rp_{2}}{\partial x^{k_{1}}\partial y^{k_{2}}\partial z^{k_{3}}}.
        \end{equation*}
  \item The \emph{generalized cross product}~\eqref{eq:cross-product} $\bS^{1}\times\bS^{2}$ corresponds to the polynomial
        \begin{equation*}
          \rp = \frac{1}{n_{1}n_{2}}\det(\xx,\nabla \rp_{1}, \nabla \rp_{2}),
        \end{equation*}
        where $\nabla \rp$ is the gradient of $\rp$.
  \item The \emph{harmonic product}~\eqref{eq:harmonic-product} $\bH^{1} \ast \bH^{2}$ corresponds to the polynomial
        \begin{equation*}
          \rp = (\rp_{1}\, \rp_{2})^{\prime}.
        \end{equation*}
\end{itemize}

Using these operations and the Cartan map~\eqref{eq:Cartan-map}, we can translate the transvectants as binary operations between tensors. In the following proposition we have made no difference between an harmonic tensor $\bH$ and its polynomial counterpart (which is an abuse of notation). Moreover, the trace of a symmetric tensor of order $n$ is defined as the contraction between any two indices.

\begin{prop}\label{prop:trad-transvectants}
  Let $\bF \in \HH^{p}(\CC^{3})$ and $\bG \in \HH^{q}(\CC^{3})$ be two harmonic tensors and set $\ff := \phi^{*}\bF$ and $\bg := \phi^{*}\bG$. Then we have
  \begin{equation}\label{eq:even-order-transvectant}
    \trans{\ff}{\bg}{2r} =  2^{-r}\phi^\ast(\bF \symrdots{r}\bG)^{\prime}
  \end{equation}
  and
  \begin{equation}\label{eq:odd-order-transvectant}
    \trans{\ff}{\bg}{2r+1} = \kappa(p,q,r) \phi^\ast (\tr^r(\bF \times \bG))^{\prime}
  \end{equation}
  where
  \begin{equation*}
    \kappa(p,q,r) = \frac{1}{2^{2r+1}} \frac{(p+q-1)! (p-r-1)!(q-r-1)!}{(p+q-1-2r)! (p-1)! (q-1)!}.
  \end{equation*}
\end{prop}

Besides polynomial invariants, one can also define rational invariants for a given representation $\VV$ of a group $G$. These are defined as rational functions on $\VV$, which are invariant under the action of $G$. These functions form a field, the field of rational invariants and is noted $K(\VV)^{G}$. An important result is the following theorem which is a corollary of a more general result due to Popov and Vinberg~\cite[Theorem 3.3]{SPV1994} (see also~\cite[Page 16]{Bri1996}).

\begin{thm}\label{thm:invariant-field}
  Let $\VV$ be a linear representation of $G$, where $G$ is either $\SL(2,\CC)$, $\SO(3,\CC)$ or $\SO(3,\RR)$ and the base field $K$ is either $\RR$ or $\CC$. Then the field of rational invariants $K(\VV)^{G}$ is the field of fractions of the invariant algebra $K[\VV]^{G}$. In other words, any rational invariant $k$ can be written as $P/Q$ where $P$ and $Q$ \emph{belong to $K[\VV]^{G}$}.
\end{thm}

A finite system of rational invariants $\mathcal{S} = \set{k_{1}, \dotsc ,k_{N}}$ generates the field $K(\VV)^{G}$ if any rational invariant $k \in K(\VV)^{G}$ can be written as a rational expression in $k_{1}, \dotsc ,k_{N}$.

\begin{rem}
  A remarkable fact is that a finite system $\mathcal{S}$ of rational invariants generates the field $K(\VV)^{G}$ if and only if $\mathcal{S}$ is a weak separating set (see~\cite[Lemma 2.1]{SPV1994}.
\end{rem}

Note that Theorem~\ref{thm:invariant-field} allows to translate any generating set of $K(\Sn{8})^{\SL(2,\CC)}$ into a generating set of $K(\HH^{4})^{\SO(3,\RR)}$.

% ----------------------------------------------------------------
\section{Maeda Invariants}
\label{sec:Maeda-invariants}
% ----------------------------------------------------------------

A minimal generating set of 9 generators for the invariant algebra of $\Sn{8}$ is known since at least 1880 (see~\cite{vGal1880,Shi1967}). In 1990~\cite[Theorem B]{Mae1990}, Maeda produced a system of 6 rational invariants which generate the invariant field $\CC(\Sn{8})^{\SL(2,\CC)}$.

\begin{thm}[Maeada, 1990]\label{thm:Maeda}
  The invariant field of binary octavics over $\CC$ is generated by the following six algebraic independent rational functions
  \begin{equation*}
    \begin{aligned}%
      I_{2}^M & := \trans{\btheta}{\btheta}{2}/M, \qquad I_{3}^M := \trans{\btheta^{3}}{\bt}{6}/M^{2}, \qquad I_{4}^M := \trans{\btheta^{4}}{\trans{\bt}{\bt}{2}}{8}/M^{3}
      \\
      J_{2}^M & := \trans{\trans{\btheta}{\ff}{1}}{\trans{\bt}{\bt}{2}}{8} \, \trans{\btheta^{6}}{\bj}{12}/M^{6},
      \\
      J_3^M   &
      := \left(36 \trans{\btheta^{2}\ff}{\bj}{12}/M^{2}  - 28 \trans{\trans{\btheta^{2}}{\ff}{3}}{\bt}{6}/{5 M} \right) \, \trans{\btheta^{6}}{\bj}{12}/M^5,
      \\
      J_4^M   & :=2 \trans{\ff \btheta^{3}}{\bt \trans{\bt}{\bt}{2}}{14}/M^{3}
      + 20 \trans{\trans{\ff}{\btheta^{3}}{1}}{\bj}{12}/ {7M^{3}}
      \\
              & \quad   -70 \trans{\trans{\ff}{\btheta^{3}}{4}}{\bt}{6}/99M^{2},
    \end{aligned}
  \end{equation*}
  where $\ff \in \Sn{8}$ is a binary form and
  \begin{gather*}
    \bQ := \trans{\ff}{\ff}{6}, \quad \bt := \trans{\trans{\bQ}{\bQ}{2}}{\bQ}{1}, \quad \btheta := \trans{\ff}{\bt}{6},
    \\
    M := \trans{\bt}{\bt}{6}, \quad \bj := \trans{\trans{\bt}{\bt}{2}}{\bt}{1}.
  \end{gather*}
\end{thm}

\begin{rem}
  We found a few minor numerical errors in~\cite{Mae1990} and did the following corrections, which were used in Theorem~\ref{thm:Maeda}.
  \begin{itemize}
    \item In~\cite[Lemma 2.10(3)]{Mae1990}, we should read
          \begin{equation*}
            \trans{\bt}{\trans{\bt}{\bt}{2}}{1} = -\bj = \Delta^3\lambda^3/108;
          \end{equation*}
    \item In~\cite[Lemma 2.12]{Mae1990}, we should read
          \begin{equation*}
            \lambda^6 \nabla = -108 \trans{\btheta^6}{\bj}{12}\Delta^3/\lambda^3;
          \end{equation*}
    \item In~\cite[Lemma 2.13]{Mae1990}, we should read
          \begin{align*}
            \lambda J_{2}/\nabla & = 72 \trans{\trans{\btheta}{\ff}{1}}{\trans{\bt}{\bt}{2}}{8}/\Delta \lambda^2,
            \\
            {J_3}/{\nabla}       & = {108}\trans{\btheta^2 \ff}{\bj}{12}/{\lambda^5}-{28}\trans{\trans{\btheta^2}{\ff}{3}}{\bt}{6}/{5\lambda^3},
            \\
            J_4                  & = 54 \trans{\btheta^3 \ff}{\bt \trans{\bt}{\bt}{2}}{14}/ \lambda^6 + 540 \trans{\trans{\ff}{\btheta^3}{1}}{\bj}{12} / 7 \lambda^6
            \\
                                 & \quad  - 70 \trans{\trans{\ff}{\btheta^3}{4}}{\bt}{6} /11 \lambda^4.
          \end{align*}
  \end{itemize}
\end{rem}

Let $\bH \in \HH^{4}$ and $\ff= \phi^{*}\bH$, the corresponding binary form of degree 8, where $\phi^{*}$ has been defined in~Appendix~\ref{sec:rational-invariants}. Using transvectants' translations obtained in Proposition~\ref{prop:trad-transvectants}, we can recast Maeda's invariants of $\ff$ as rational invariants of $\bH$. We get first
\begin{equation*}
  \begin{aligned}
    \phi^{-*} \bQ     & = \phi^{-*} \trans{\ff}{\ff}{6} = \frac{1}{8} \bd_{2}^{\prime},
    \\
    \phi^{-*} \bt     & = \phi^{-*} \trans{\trans{\bQ}{\bQ}{2}}{\bQ}{1} = \frac{1}{2^{11}} \bd_{2}^{\,2}\times \bd_{2} = \frac{1}{2^{11}} \bT_{6},
    \\
    M                 & = \phi^{-*} \trans{\bt}{\bt}{6} = \frac{1}{2^{25}} \norm{\bT_{6}}^{2} = \frac{1}{2^{25}} M_{12},
    \\
    \phi^{-*} \btheta & = \phi^{-*} \trans{\ff}{\bt}{6} = \frac{1}{2^{14}} \ww_{7} = \frac{1}{2^{14}} \bH \3dots \bT_{6},
    \\
    \phi^{-*} \bj     & = \phi^{-*} \trans{\trans{\bt}{\bt}{2}}{\bt}{1} = \frac{1}{2^{35}}\left((\bT_{6} \symrdots{1} \bT_{6})^{\prime} \times \bT_{6}\right)^{\prime} = \frac{1}{2^{35}} \bJ_{18},
  \end{aligned}
\end{equation*}
where $\phi^{-*}$ stands for the inverse of $\phi^{*}$ and where we have used the following observations.
\begin{enumerate}
  \item If $\bH \in \HH^{n}(\RR^{3})$ and $\bq$ is the Euclidean tensor, then,
        \begin{equation*}
          (\odot^{k}\bq) \times \bH  = 0, \qquad \forall k \ge 1,
        \end{equation*}
        where $\odot^{k}\bq$ is the symmetric tensor product of $k$ copies of $\bq$.
  \item If $\bH \in \HH^{n}(\RR^{3})$ and $\ww \in \HH^{1}(\RR^{3})$, then, $\ww \times \bH$ is harmonic.
  \item If $\ba\in \Sym^{2}(\RR^{3})$, then ${\ba}^{2} \times \ba$ is harmonic (see~\cite[Remark 8.2]{OKDD2018a}) and
        \begin{equation*}
          {\ba}^{2} \times \ba = {\ba^{\prime}}^{2} \times \ba^{\prime}.
        \end{equation*}
  \item If $\bT^{1}, \bT^{2} \in \TT^{n}(\RR^{3})$, then, $\bT^{1} \rdots{n} \bT^{2} = \langle \bT^{1}, \bT^{2}\rangle$
        is their scalar product and
        \begin{equation*}
          \langle \bT_{1}, \bT_{2}^{s} \rangle = \langle \bT_{1}^{s}, \bT_{2} \rangle, \qquad \langle \bT_{1}, (\bT_{2}^{s})^{\prime} \rangle = \langle (\bT_{1}^{s})^{\prime}, \bT_{2} \rangle.
        \end{equation*}
\end{enumerate}

We get then
\begin{equation*}
  \phi^{-*}\trans{\btheta^{2}}{\ff}{3} = \frac{5}{6}\tr [ (\ww_{7}\ast\ww_{7})\times \bH] = -\frac{1}{4}\, (\bH\cdot\ww_{7})\times \ww_{7},
\end{equation*}
which is an harmonic third-order tensor, by (2) and the fact that $\bH\cdot\ww_{7}$ is itself harmonic. We have finally the following result, where we have introduced the notation $\ast^{k}\, \ww_{7}$ for the harmonic product of $k$ copies of $\ww_{7}$. We point out, moreover, that the first-order covariant $\ww_{7}$, the third-order covariant $\bT_{6}$ as well as the sixth-order covariant $\bJ_{18}$ are all harmonic.

\begin{thm}\label{thm:Maeda-H4}
  The invariant field of $\HH^{4}(\RR^3)$ is generated by the following six algebraic independent rational functions
  \begin{align*}
    i_{2} & =  \frac{\norm{\ww_{7}}^{2}}{M_{12}},
    \\
    i_{3} & =  \frac{\langle \ast^{3}\, \ww_{7}, \bT_{6}\rangle }{M_{12}^{2}},
    \\
    i_{4} & = \frac{\langle \ast^{4}\, \ww_{7}, \bT_{6}\cdot \bT_{6}\rangle }{M_{12}^{3}}
    \\
    k_{4} & = \frac{1}{5 M_{12}^{3}} \langle \bH \ast (\ast^{3}\, \ww_{7}), \bT_{6}\ast (\bT_{6} \symrdots{1} \bT_{6})^{\prime}\rangle
    + \frac{1}{7 M_{12}^{3}} \langle \bH \times (\ast^{3}\, \ww_{7}), \bJ_{18}\rangle
    \\
          & \quad -\frac{7}{99 M_{12}^{2}} \langle \bH : (\ast^{3}\, \ww_{7}),   \bT_{6}\rangle.
    \\
    k_{8} & =  \frac{\langle \ww_{7} \times \bH, \bT_{6}\cdot \bT_{6}\rangle \, \langle \ast^{6}\, \ww_{7}, \bJ_{18} \rangle }{M_{12}^{6}},
    \\
    k_9   & = \frac{\langle \ast^{6}\, \ww_{7}, \bJ_{18} \rangle }{M_{12}^5} \left( \frac{36}{M_{12}^{2}} \langle (\ast^{2}\, \ww_{7}) \ast \bH, \bJ_{18} \rangle + \frac{28}{5M_{12}}\langle (\bH\cdot\ww_{7})\times\ww_{7}, \bT_{6}\rangle \right),
  \end{align*}
  where $\bH\in \HH^{4}(\RR^{3})$ is the harmonic tensor, and
  \begin{align*}
     & \bT_{6} := {\bd_{2}}^{2}\times  \bd_{2}, &  & M_{12}  := \norm{{\bd_{2}}^{2} \times \bd_{2}}^{2},
    \\
     & \ww_{7} := \bH \3dots \bT_{6},           &  & \bJ_{18} := (\bT_{6} \symrdots{1} \bT_{6})^{\prime} \times \bT_{6}.
  \end{align*}
\end{thm}

% ----------------------------------------------------------------

\end{document}